\documentclass[a4paper] {article}

\usepackage{amscd}
\usepackage{amssymb}
\usepackage{amsmath}
\usepackage{amsthm}

\usepackage{fancyhdr}

\newtheorem{theorem}{Theorem}[section]

\newtheorem{lemma}[theorem]{Lemma}
\newtheorem{corollary}[theorem]{Corollary}
\newtheorem{proposition}[theorem]{Proposition}

\newtheorem*{remark_}{Remark}

\newtheorem*{remarks_}{Remarks}
\newtheorem*{theoremA}{Theorem A}
\newtheorem*{theoremB}{Theorem B}
\newtheorem*{lemmaC}{Lemma C}
\newtheorem*{theoremD}{Theorem D}

\theoremstyle{definition}
\newtheorem{example}[theorem]{Example}

\numberwithin{equation}{section}

\newcommand{\spa}{\operatorname{span}}

\newcommand{\Ran}{\operatorname{Ran}}
\newcommand{\rank}{\operatorname{rank}}

\newcommand{\clos}{\operatorname{clos}}
\newcommand{\limst}{\operatorname{lim}^*}
\newcommand{\slw}{\operatorname{Slow}}

\newcommand{\tht}{\theta}   %       !!!!!!!!!!!!!!!!!!!!!!!!!!!!!!!!

\newcommand\defin {\overset {\text {\rm def} }{=}}
\newcommand\precd {\overset {d}{\prec}}
\newcommand{\A}{\mathcal{A}}
\newcommand{\LL}{\mathcal{L}}
\newcommand{\T}{\mathcal{T}}

\newcommand\beqn{\begin{equation}}
\newcommand\neqn{\end{equation}}

\newcommand{\he}{\mathcal{H}(E)}

\newcommand{\al}{\alpha}   %       !!!!!!!!!!!!!!!!!!!!!!!!!!!!!!!!
\newcommand{\be}{\beta}   %       !!!!!!!!!!!!!!!!!!!!!!!!!!!!!!!!
\newcommand{\de}{\delta}   %       !!!!!!!!!!!!!!!!!!!!!!!!!!!!!!!!
\newcommand{\la}{\lambda}
\newcommand{\eps}{\varepsilon}
\newcommand{\si}{\sigma}   %       !!!!!!!!!!!!!!!!!!!!!!!!!!!!!!!!
\newcommand{\wt}{\widetilde } %       !!!!!!!!!!!!!!!!!!!!!!!!!!!!!!!!
\newcommand{\cO}{\mathcal{O} } %       !!!!!!!!!!!!!!!!!!!!!!!!!!!!!!!!
\newcommand{\cD}{\mathcal{D}}

\newcommand{\BR}{\mathbb{R} } %       !!!!!!!!!!!!!!!!!!!!!!!!!!!!!!!!
\newcommand{\BN}{\mathbb{N} }
\newcommand{\NN}{\mathbb{N} }

\newcommand{\BZ}{\mathbb{Z} }
\newcommand{\RR}{\Bbb R}
\newcommand{\BC}{\Bbb C}
\newcommand{\sm}{\setminus}
\renewcommand{\phi}{\varphi}
\renewcommand{\Re}{\operatorname{Re}}
\renewcommand{\Im}{\operatorname{Im}}

\newcommand{\cDA}{\cD(\A)}
\newcommand{\cDL}{\cD(\LL)}

\newcommand{\an}{\mathbb{A}}
\newcommand{\BT}{\mathbb{T}}

\begin{document}

% \sloppy

\title{Completeness of  rank one perturbations
of normal operators with lacunary spectrum}

% SHORT TITLE:    [Completeness of perturbations of normal operators]

\author{Anton D. Baranov, Dmitry V. Yakubovich}
% \author{}

\pagestyle{fancy}
\lhead{\small A.D. Baranov, D.V. Yakubovich}
\rhead{\small Completeness of perturbations of normal operators}

\maketitle

\begin{abstract}
Suppose $\A$ is a compact normal operator
on a Hilbert space $H$ with certain lacunarity
condition on the spectrum (which means, in particular, that
its eigenvalues go to zero exponentially fast), and let $\LL$ be
its rank one perturbation. We show that either infinitely many
moment equalities hold or the linear span of root vectors of $\LL$, corresponding
to non-zero eigenvalues, is of finite codimension in $H$. In contrast to classical
results, we do not assume the perturbation to be weak.

\medskip

\noindent {\bf M.S.C.(2000):}  Primary: 42A65; Secondary: 42C30

\noindent {\bf Keywords:} selfadjoint operator, rank one perturbation,
completeness of eigenvectors, P\'olya peaks
\end{abstract}

{}
\vskip-2cm

\section{Introduction and main results}

\subsection{The main result}

Let $\LL$ be a compact operator on a separable Hilbert space $H$.
We will say that $\LL$ is \textit{complete} if its root vectors,
corresponding to non-zero
eigenvalues, are complete in $H$.
(Notice that the point $0$ of the spectrum of a compact operator plays a special role.)
We will say that $\LL$ is \textit{nearly complete} if
the root vectors of $\LL$, corresponding to non-zero
eigenvalues,  span a subspace of $H$ of finite codimension.
One can observe that for any positive integer $N$, the closed linear span
$M(\LL)$ of root vectors of $\LL$, corresponding to non-zero eigenvalues,
is contained in $\big(\ker \LL^{*N}\big)^\perp$.
So, whenever $\ker \LL^*\ne 0$, $\LL$ cannot be complete and we only can expect the
near completeness.

If one wishes to apply the Fourier method to a linear evolution equation
associated with a non-normal compact operator $\LL$, the first obstacle is that eigenvectors of $\LL$
need not form an orthogonal basis. The completeness of $\LL$ is probably the weakest
substitute of this property. The strongest
one in the non-normal case is the Riesz basis property, and there is a whole range
of intermediate properties, related with availability of different linear summation
methods, such as Ces\`aro or Abel summability.
Most general abstract sufficient conditions for completeness are
due to Keldy\v s \cite{Keldys, Keldys71} and Macaev \cite{Mats61, Mats64}.
A good exposition of these results by Macaev and
their generalizations to operator pencils can be found in
\cite{MatsMog71}.

\begin{theoremA} [Keldy\v s, 1951]
Let $A$, $S$ be compact Hilbert space operators. Suppose $A$ is
normal, belongs to a Schatten ideal $\mathfrak {S}_p$,
$0<p<\infty$, and its spectrum is contained in a finite union of
rays $\arg z=\al_k$, $1\le k\le n$. Suppose $\ker A=\ker(I+S)=0$.
Put $L=A(I+S)$. Then the operators $L$ and $L^*$ are complete.
\end{theoremA}

The original statement by Keldy\v s referred only to the case of a
selfadjoint operator $A$; the above formulation appears, for
instance, in \cite{markus-book}.
A perturbation of a compact operator $A$ of the form $A(I+S)$ or
$(I+S)A$, with $S$ compact, is called a {\it weak perturbation}.
Macaev's theorems also concern weak perturbations.
In \cite{MatsMogul76}, Macaev and Mogul'skii give an explicit
condition on the spectrum of $A$, equivalent to the property that
all weak perturbations of $A$ with $\ker(I+S)=0$ are complete.

Our results deal with the
situation which is much more special than in the Keldy\v s and
Macaev theorems: namely, we consider one-dimensional perturbations
of normal operators, which are not
necessarily weak, and also treat the case of nontrivial kernels.
In this case the spectral theory of one-dimensional perturbations
of normal (even of selfadjoint) operators
becomes rich and complicated (see \cite{baryak, bbby} where a functional
model for such perturbations is constructed).

Note that one of the basic methods to prove the near completeness
of an operator is to obtain growth estimates of its resolvent and
then to apply some appropriate results from the theory of entire
functions. This is essentially the method used by Keldy\v s and
Macaev (for a general statement of this type see \cite[Theorem
XI.9.29]{Dunf_Schwarz}). In the present paper we consider
one-dimensional perturbations of normal operators with very sparse
(lacunary) spectra. Using the resolvent estimates we will show
that in this case a one-dimensional perturbation is nearly
complete unless it has strong degeneracy.
This phenomenon is related to the lacunarity of the spectrum and does
not appear in general.

Let $\{\la_n\}$ be a sequence of complex numbers. We will say that
this sequence is \textit{lacunary} if there is a positive constant
$\eps$ such that $|\la_n-\la_m|\ge \eps \max(|\la_m|,|\la_n|)$ for
all indices $n\ne m$. This is equivalent to the condition that for
some $\de>0$, the discs $B(\la_n, \de |\la_n|)$ are disjoint. Such
sequence can accumulate only to $0$ and to $\infty$. If $\la_n$
tend to zero as $n\to+\infty$ and are numbered so that the moduli
$|\la_n|$ decrease, then they decay exponentially fast, moreover,
there exist some $\si\in (0,1)$ and some $B\ge 1$ such that
$|\la_m/\la_n| \le B \si^{m-n}$ whenever $m\ge n$.

Suppose that $\A$ is a compact normal operator with trivial kernel. By the Spectral Theorem,
\beqn
\label{sp-decompA}
\A=\sum_{n\in \BN} s_n P_n,
\neqn
where $\BN=\{1,2,\dots\}$, $s_n\ne 0$, $s_n\to 0$, and $P_n$ are finite dimensional orthogonal projections
in $H$ such that $P_nP_m=0$ for $m\ne n$ and $\sum_n P_n=I$.
Our main object of study will be a one-dimensional perturbation of $\A$
of the following form:
\beqn
\label{LL}
\LL x = \A x + \langle x, b \rangle\, a,
\qquad a,b\in H.
\neqn
To formulate our results, we need to introduce
the following sequence of ``moment'' equations
\[
\label{seqs-eqs}
\begin{aligned}
\sum_n s_n^{-1} \langle P_n a, b\rangle = -1, \quad \qquad\qquad  ({\rm M }_1) \\
\sum_n s_n^{-k} \langle P_n a, b\rangle =0,
\quad \qquad\qquad\,  ({\rm M }_k)
\end{aligned}
\]
$k=2,3, \dots$.
 Note that, for a general one-dimensional perturbation,
the above series need not converge.

Our first main result is Theorem~\ref{3to1} below. It might be
instructive to precede it with two simpler statements.

\begin{theorem}
\label{thm0angles} Let $\A$ be a normal operator given
by~\eqref{sp-decompA} which belongs to a Schatten ideal $\mathfrak
{S}_p$, $0<p<\infty$, and whose spectrum is contained in a finite
union of rays $\arg z=\al_\ell$, $1\le \ell\le n$. Let $\LL$ be a
one-dimensional perturbation of $\A$, given by \eqref{LL}. Assume
that, for some $k\in \mathbb{N}$, we have
$$
\sum_n |s_n|^{-k} |\langle P_n a, b\rangle| <\infty,
$$
but the equality $({\rm M}_k)$ does not hold.
Then $\LL$ and $\LL^*$ are nearly complete.

Moreover, for any $\eps >0$ there is a radius $r>0$ such that the
intersection of the non-zero spectrum $\si(\LL)\sm\{0\}$ with the disc
$B(0,r)$ is contained in the union of angles $\al_\ell-\eps<\arg
z<\al_\ell+\eps$, $1\le \ell \le n$.
\end{theorem}

This assertion  can be obtained by an application of standard
methods based on resolvent estimates, see Section~\ref{prelim}.

Another simple observation is that $\cup_n \ker \LL^{*n}$ is
orthogonal to $M(\LL)$. So, if the linear manifold
$\cup_n \ker \LL^{*n}$ is infinite dimensional, then,
obviously, $\LL$ is not nearly complete. It is easy to see that
the following fact holds.

\begin{proposition}
\label{prop1}
$\cup_n \ker \LL^{*n}$ is infinite dimensional if and
only if $b\in \Ran \A^n $ for any integer $n>0$ and the equalities
$({\rm M }_k)$ hold for all $k\ge 1$.
\end{proposition}

The next theorem shows that a stronger statement than Theorem~\ref{thm0angles} holds for the
case of a \textit{lacunary} spectrum (with arbitrary geometry).
At the same time, it can be seen as a partial converse of Proposition~\ref{prop1}.

\begin{theorem}
\label{3to1}
Let $\LL$ given by \eqref{LL}  be a one-dimensional
perturbation of a compact normal operator $\A$,
given by~\eqref{sp-decompA}, whose spectrum
is lacunary. If $\LL$ is not nearly complete, then the
equalities $({\rm M }_{k})$ are valid for all $k\ge 1$.
\end{theorem}

\begin{remarks_}
{\rm 1. Since $\LL^*x=\A^* x + \langle x,a\rangle b$, it follows that the same criterion holds
for near completeness of $\LL^*$, where
equalities $({\rm M }_{k})$ are literally the same.
\smallskip

2. In the case when all moment equalities
$(M_k), k\ge 1$, are fulfilled, the operator $\LL$ may be complete or incomplete.
In \cite[Theorem 1.3]{baryak},
given any compact selfadjoint operator $\A$
with simple point spectrum and trivial kernel,
a bounded rank one perturbation $\LL$ of $\A$ with real spectrum was constructed
such that $\LL$ is complete and $\ker \LL = 0$, while $\LL^*$ is even not nearly complete.
Therefore the near completeness of
$\LL$ is not equivalent to the near completeness of
$\LL^*$, even for rank one perturbations of normal operators with lacunary spectrum
we are considering here.

It is essential for the construction in \cite{baryak} that all
moment equalities hold. For the lacunary spectra this follows also
from our Theorem \ref{3to1}.

\smallskip

3. It would be interesting to know whether an analogue of Theorem
\ref{3to1} holds true for finite rank perturbations.
It also would be interesting to know whether
in the conclusion of this theorem, one can replace the completeness with
the spectral synthesis property.
\medskip
}

\end{remarks_}

%%%%%%%%%%%%%%%%%%%%%%%%%%%%%%%%%%%%%%%%%%%%%%%%%%%%%%%%%%%%%%%%%%%%%

\subsection{Sharpness of Theorem \ref{3to1}}
Our second main result says that the lacunarity condition in
Theorem \ref{3to1} cannot be weakened much.
Namely, if the spectrum of $\A$ is at least slightly more dense than
a lacunary one, then there exists a rank one perturbation, which is not
nearly complete, but already the first moment does not exist,
moreover,
\begin{equation}
\label{beg}
\sum_n |s_n|^{-1} |\langle P_n a, b\rangle| = \infty.
\end{equation}
To make the conditions on the spectrum clearer it is better to
pass to the inverses $t_n =s_n^{-1}$.
Note that the lacunarity implies that $n_T(r) = O(\log r)$, $r\to \infty$.
Here $n_T$ is the counting function of the sequence $\{t_n\}$:
$n_T=\#\{n: |t_n|<r\}$.
We show that rank one perturbations satisfying \eqref{beg}, which fail to be nearly
complete, always exist
unless $n_T(r) = O(\log^2 r)$, $r\to \infty$.

The precise formulation of our second main result is as follows:

\begin{theorem}
\label{hhn}
Let $\A_0$ be a compact selfadjoint operator, which has a
representation of the form~\eqref{sp-decompA}, where $s_n\in \mathbb{R}$, $s_n\ne 0$.
Suppose that $\rank P_n=1$ for all $n$. Put $t_n = s_n^{-1}$.
Assume
that $\inf_{n\ne k} |t_n-t_k| >0$ and that, for any $N>0$, we have
\begin{equation}
\label{beg1} \liminf_{|n|\to\infty} |t_n|^N \prod_{k: \,
\frac{1}{2} \le \frac{t_k}{t_n} \le 2, \, k\ne n}
\bigg|\frac{t_k-t_n}{t_k}\bigg| =0.
\end{equation}
Then there exists a rank one perturbation $\LL_0=\A_0 + \langle \cdot , b\rangle\, a$ of $\A_0$
such that \eqref{beg} holds,
but
$\LL_0$ is not nearly complete.
It is true, in particular, if
\beqn
\label{beglog2}
\limsup_{r\to\infty} \, \frac {n_T(r)}{\log^2 r}=+\infty.
\neqn
\end{theorem}

We will show, in fact, that \eqref{beglog2} implies \eqref{beg1}
(see Corollary~\ref{4to5}).
One can express the condition \eqref{beg1} in equivalent ways, see
the remark in Subsection \ref{prelim1}.

%%%%%%%%%%%%%%%%%%%%%%%%%%%%%%%%%%%%%%%%%%%%%%%%%%%%%%%%%%%%%%%%%%%%%%%%%%

\subsection{Methods of the proof}

This work can be considered as a continuation of our papers
\cite{baryak, baryak2}, where the completeness and related
properties (e.g., the spectral synthesis) were studied for similar
class of perturbations of selfadjoint operators without an assumption on the
lacunarity of the spectrum.
%... were studied for similar class of operators by using a kind of functional model involving de Branges spaces. In these works, we considered perturbations of
%selfadjoint operators  without a condition on lacunarity of the spectum. Here we consider perturbations of general compact normal
%operators with lacunary spectrum.
As in \cite{baryak, baryak2}, here we also will consider
(singular) rank one perturbations of unbounded normal operators
with discrete spectrum and obtain parallel completeness results
for them.

The spectral analysis of
the perturbation \eqref{LL} of operator $\A$
leads to a consideration of the function
\beqn
\label{beta}
\be(z)=
1 + \langle (\A-z^{-1})^{-1} a,b\rangle =
 1 - z \langle (I-z \A)^{-1} a,b\rangle = 1+ \sum_n c_n \Big(\frac 1 {t_n-z}  - \frac 1 {t_n} \Big),
\neqn
where $t_n=1/s_n$ and $c_n= - s_n^{-2} \langle P_n a, b\rangle$. This function
is meromorphic in $\BC$.
It is easy to see that the zero set of $\be$ coincides with the set
$\{\la^{-1}: \la\in \si(\LL), \la\ne0\}$.

We will adopt the following notation.
If $A$ is a measurable subset of $[0,+\infty)$, its
\textit{linear density}
is defined as $\lim_{R\to+\infty} R^{-1}m([0,R]\cap A)$.
Given a function
$f$ on $\BC$, we will write $\limst_{z\to \infty} f(z) = w$ if
there exists a closed set $A \subset [0,+\infty)$ of linear
density one such that
\[
\lim_{z\to \infty,\, |z|\in A} f(z)=w.
\]
Our main complex variable tool
for proving Theorem~\ref{3to1}
will be the following statement.

\begin{theorem}
\label{comp-var-thm1}
Suppose that a complex sequence $\{t_n\}$
goes to infinity, is lacunary and $t_n\ne 0$ for all $n$.
Let $\varkappa\in \BC$ and let $c_n$ be any complex coefficients, not all equal to zero, such that
$\sum_n |c_n/t_n^{2}| < + \infty$. Put
\beqn
\label{beta2}
\be(z)=\varkappa+ \sum_n c_n \Big(\frac 1 {t_n-z}  - \frac 1 {t_n} \Big).
\neqn
If for any $s\in \BN$,
\beqn
\label{limst2}
\limst_{z\to \infty}  z^s \be(z) =0,
\neqn
then $\sum_n t_n^{-1} c_n= \varkappa$ and
$\sum_n t_n^k c_n=0$ for $k\in \BZ$, $k\ge 0$.
\end{theorem}

It is easy to see that the conditions $\sum_n t_n^{-1} c_n=
\varkappa$ and $\sum_n t_n^k c_n=0$, $k\ge 0$, imply that $\lim z^s
\beta(z) =0$ as $|z|\to \infty$ for any $s$ and ${\rm dist}\, (z,
\{t_n\}) \ge 1$. Theorem~\ref{comp-var-thm1} shows that in the
lacunary case the converse is true. The proof of this theorem uses
a lemma on ``peaks'' of numerical sequences by P\'olya
\cite{Polya23}, which enables us to obtain some estimates of the
function $\be$ from below.

Let us mention that in \cite{Shk_1984}, Shkalikov obtained lower
estimates for meromorphic functions outside small ``exceptional''
sets, which he applied in \cite{Shk_2010} to get a new criterion
for eigenvectors of a perturbed selfadjoint operator to form a
basis with parenthesis.

\medskip

The proof of Theorem~\ref{3to1}
is based on the above Theorem~\ref{comp-var-thm1} and
deals directly with resolvent estimates.
On the other hand,
an important tool in proving Theorem \ref{hhn} is a passage to the functional model
of $\LL$ in a so-called de Branges space $\he$ of entire functions, associated to an entire
function $E$ in the Hermite--Bieler class (see Subsection~\ref{funmod} below).
The completeness property for
$\LL^*$ turns out to be equivalent to completeness of a certain system of
reproducing kernels in $\he$. This allows us to apply complex analysis tools, such as
factorizations of entire functions and partial fraction expansions of their quotients.

De Branges spaces and the related model spaces associated to meromorphic inner functions
have numerous and well-known appplications to the
spectral theory of discrete selfadjoint operators. For
instance, Makarov and Poltoratski apply this approach in~\cite{mak-polt} to get broad extensions of
several classical results
such as Borg's two spectra theorem and the Hochstadt-Lieberman theorem
concerning the unique determination of Schr\"odinger operators
on a finite interval by their spectra. In
\cite{Silva-Teschl-Toloza}, de Branges spaces were applied to model
a subclass of singular Schr\"odinger operators.
In our previous papers \cite{baryak, baryak2}, we applied the model approach to
the completeness of nonselfadjoint rank one perturbations
of a discrete selfadjoint operator.
These papers also exploit the property of spectral synthesis, which is weaker than
the existence of whatever linear summation method for expansions in eigenvectors.
Completeness problems of ''mixed'' systems were treated in
recent papers \cite{bel-men-seip, bb, bbb, bbb1}.
In particular, in \cite{bel-men-seip}, Riesz bases of reproducing
kernels in spaces of Cauchy transforms of discrete lacunary
measures in $\BC$ have been described. It seems that there are some
relations between this work and the questions we study here.

We also wish to mention here papers~\cite{Martin}
and~\cite{Silva-Toloza}, where de Branges spaces were applied to
the spectral theory in a context close to the context of
\cite{baryak, baryak2}.
\medskip

A result on completeness or near completeness of an abstract
compact operator in a Schatten class $\mathfrak {S}_p$ is
contained in \cite[Theorem XI.9.29]{Dunf_Schwarz}, where power
estimates of the resolvent in certain angles are assumed.
We
remark that, although Keldy\v s and Macaev's theorems can be
deduced from this theorem, the condition $\ker(I+S)=0$ is crucial
for getting these power estimates (see the Example~\ref{ex}
below).

We refer to \cite{AddMit, Shk_2010, Wyss}
for some recent abstract results about completeness and bases
of eigenvectors and to the books \cite{Gohb_Krein, Gohb_Kr_Volterr, Dunf_Schwarz,
markus-book} and the review \cite{Radz82} for an extensive exposition.
The papers \cite{Ionascu, FoiasJungPearcy, FangXia}
contain some general results on spectral properties of finite dimensional
perturbations of diagonal operators (in particular, on the existence of
invariant subspaces).
It seems that not much is known in general for weak perturbations with
$\ker(I+S)\ne0$ and for non-weak perturbations.

The paper is organized as follows. In Section \ref{prelim} we collect some preliminary
results relating the near completeness to properties of some entire (meromorphic)
functions defined in terms of resolvents. Section \ref{mainpr} contains the proof of
our main result, Theorem \ref{3to1}. Some additional remarks on the general
properties of nearly complete operators are given in Section \ref{addit}.
In Section \ref{singul} we introduce (unbounded) singular rank one perturbations
and state the counterpart of the main result for this case. Theorem \ref{hhn}
about the necessity of some (slightly weaker than lacunarity) sparseness
condition is proved in Section \ref{sharpr} by an application of a functional model.
\medskip

Selfadjoint operators with lacunary spectrum are not so
seldom in practice. In Section \ref{addit}, we give examples of
convolution operators on $L^2$, which are normal, compact and
lacunary. In Section \ref{singul}, we comment on Jacobi matrices,
which define unbounded selfadjoint lacunary operators. They arise,
in particular, in relation with $q$-harmonic oscillator and $q$-orthogonal polynomials.
\medskip
\\
\textbf{Acknowledgements.}
The authors express their gratitude to R.~Romanov, G.~Ros\-en\-blum and M.~Solomyak
for helpful discussions and bibliographical remarks.
We are also indebted to the anonymous referee for valuable suggestions,
which helped us to improve the exposition.

A. Baranov was supported by the grant MD-5758.2015.1 and by RFBR
grants 13-01-91002-ANF-a and 14-01-00748-a. D. Yakubovich was
supported by the project MTM2015-66157-C2-1-P and the ICMAT Severo
Ochoa project SEV-2015-0554 of the Ministry of Economy and
Competitiveness of Spain.

\section{Preliminary results on near completeness}
\label{prelim}

We start with several simple remarks which will be used in what follows.
Note that they are true for general compact operators (not necessarily with lacunary spectra).

\begin{proposition}
\label{prop-n-compl}
Let $\LL$ be a general compact operator on a
Hilbert space $H$. If $\LL$ is nearly complete, then the linear
set  $\cup_n \ker \LL^{*n}$ is finite dimensional, there is some
$k\ge 0$ such that $\ker \LL^{*n}=\ker \LL^{*k}$ for all $n>k$ and
$M(\LL)^\perp=\cup_n \ker \LL^{*n}= \ker \LL^{*k}$.
\end{proposition}

The proof is immediate from the fact that $\LL^*|M(\LL)^\perp$ is
always quasinilpotent.

Next, let us make the following observation. Let $H^1$ be the
smallest reducing subspace of $\A$ containing the vectors $a, b$,
and put $H^2=H\ominus H^1$. Then $\LL $ decomposes as
$\LL=\LL^1\oplus \A^2$, where $\LL^1 x= \A x + \langle x, b\rangle
a$, $x\in H^1$ and $\A^2=\A| H^2$. This implies that it suffices
to prove Theorem~\ref{3to1} for the case when the smallest
reducing subspace of $\A$ containing vectors $a, b$ coincides with
$H$. From now on, we will assume that this property is fulfilled.
It follows that all spectral projections $P_n$ in
\eqref{sp-decompA} have at most rank two.

Put $M=M(\LL)$. Given any $x\in H$ and any $y\in M^\perp$, the
function
\[
f_{x,y}(\la)=\langle (I-\la \LL)^{-1}x, y\rangle
\]
is entire.

We will use an easy formula
\beqn
\label{reslv-L-A}
(\LL-z)^{-1} = (\A-z)^{-1} - (\A-z)^{-1} a \,\be^{-1}(z^{-1}) \,b^* (\A-z)^{-1},
\neqn
where $\be(z)$ is defined in \eqref{beta}.
%\marginpar{\small Tak, kak zdes' napisano, \eqref{reslv-L-A} verna i dlya возмущений ранга $n$; togda
%$\be(z)$ - matrica $n\times n$.}
To check it, one can write down the second resolvent identity
\[
(\LL-z)^{-1} - (\A-z)^{-1} = - (\A-z)^{-1} ab^* (\LL-z)^{-1}.
\]
By multiplying it by $b^*$ on the left, one gets $b^*(\LL-z)^{-1}=
\be^{-1}(z^{-1})b^*(\A-z)^{-1}$, which gives \eqref{reslv-L-A}.
Then, for any $s\in \BN$, $x\in H$ and $y\in M^\perp$
we have
\beqn
\label{id}
\begin{aligned}
z^{-s-3}  & f_{x,y} (z) \\
& = z^{-s-4} y^*(z^{-1}-\LL)^{-1} x \\
& = z^{-s-4} y^*(z^{-1}-\A)^{-1} x \\
& \qquad    +  \big[z^{-2} y^*(\A-z^{-1})^{-1}a\big] \cdot
 \, z^{-s} \be^{-1}(z) \cdot \big[z^{-2} b^*(\A-z^{-1})^{-1}x\big].
\end{aligned}
\neqn
%
%
%\beqn \label{id}
%\begin{aligned}
%z^{-s-3} f_{x,y}(z)
%& = z^{-s-4} y^*(z^{-1}-\LL)^{-1} x \\
%& = z^{-s-4} y^*(z^{-1}-\A)^{-1} x \\
%& \qquad    +  \big[z^{-2} y^*(\A-z^{-1})^{-1}a\big] \cdot
% \, z^{-s} \be^{-1}(z) \cdot \big[z^{-2} b^*(\A-z^{-1})^{-1}x\big].
%\end{aligned}
%\neqn

\begin{proposition}
\label{prop-nrl-compl}
Suppose $N\in \BZ$, $N\ge0$
and $\LL$ is a compact operator. Then the following statements are equivalent:

\begin{itemize}

\item[(i)]
$M^\perp\subset \ker \LL^{*N}$;

\item[(ii)]
For any $x\in H$ and any $y\in M^\perp$,
$f_{x,y}(\la)$ is a polynomial in $\la$ of degree less than~$N$.
\end{itemize}
\end{proposition}

\begin{proof}[Proof]
Since $f_{x,\LL^*y}(\la)=\la^{-1}[f_{x,y}(\la)-f_{x,y}(0)]$, it follows that
(ii) is equivalent to the condition
$f_{x,\LL^{*N}y}\equiv 0$ for all
$x\in H$, $y\in M^\perp$, that is, to the condition
$M^\perp\subset \ker\LL^{*N}$. This gives the statement.
\end{proof}

Suppose, in particular, that $\ker \LL^{*N}$ is finite dimensional for any
$N\in \BN$ (it is true for the operator $\LL$, given by \eqref{LL}).
Then, by the Propositions~\ref{prop-nrl-compl} and \ref{prop-n-compl},
$\LL$ is nearly complete if and only if
there is an integer $N>0$ such that
for any $x\in H$ and any $y\in M^\perp$,
$f_{x,y}(\la)$ is a polynomial in $\la$ of degree less than $N$.

\begin{proof}[Proof of Theorem~\ref{thm0angles}]
Assume that $k\in \mathbb{N}$ is the smallest positive integer
such that
\newline
$\sum_n |s_n|^{-k} |\langle P_n a, b\rangle| <\infty$
but the equality $({\rm M}_k)$ does not hold.
By the above remark, assume that $k\ge 2$.
Put $t_n=s_n^{-1}$ and define $\be(z)$ by
\eqref{beta}, where $c_n= - t_n^2 \langle P_n a, b\rangle$.
Then we have  $\sum_n |t_n|^k|c_n|<\infty $ and $\sum_n t_n^k c_n = \gamma\ne 0$.
By the obvious formula
$$
\frac{1}{z-t_n} = \sum_{j=0}^k \frac{t_n^j}{z^{j+1}} + \frac{t_n^{k+1}}{z^{k+1}(z-t_n)},
$$
we have
$$
\beta(z) =- \frac{1}{z^{k+1}}\sum_n t_n^{k}c_n  +
\frac{1}{z^{k+1}}\sum_n \frac{t_n^{k+1}c_n}{t_n-z} = \frac{\gamma}{z^{k+1}}
+\frac{1}{z^{k+1}}\sum_n \frac{t_n^{k+1}c_n}{t_n-z}.
$$
Recall that $\{t_n\}$ is contained in a finite union of rays
$\arg z=\al_\ell$, $1\le \ell\le n$. Hence, for any
$\varepsilon>0$, we have
$$
|\beta(z)| >\frac{c_\varepsilon}{|z|^{k+1}}
$$
when $|\arg z-\al_\ell|\ge \varepsilon$, $1\le \ell\le n$ and
$|z|$ is sufficiently large. Clearly, $\|(\A - z^{-1})^{-1}\|$
also admits a power estimate for such values of $z$ and we
conclude, by \eqref{id}, that $|f_{x,y}(z)|$ admits a power
estimate for $|\arg z-\al_\ell|\ge \varepsilon$. Since $\A \in
\mathfrak {S}_p$, we get that $f_{x,y}$ is a function of finite
order and, by the Phragm\'en--Lindel\"of principle, we conclude
that $f_{x,y}$ is a polynomial of degree less than some fixed $N$
for any $x\in H$, $y\in M^\perp$.
\end{proof}

\begin{remark_}
{\rm If the series in $(M_1)$ converges absolutely, then $\LL$ is what we
called in \cite{baryak} a generalized weak perturbation of $\A$.
The case when the hypotheses of Theorem~\ref{thm0angles} hold for
$k=1$ was treated in \cite[Theorem 3.3]{baryak} for rank one perturbations
of selfadjoint operators (the assumption that $\A \in \mathfrak{S}_p$ can
be dropped). In this case the completeness of a perturbation can be also
derived from the results of Macaev as in the proof of \cite[Proposition 1.1]{baryak}. }
\end{remark_}

Now we pass to the analysis of the case when the spectrum is lacunary.

Given an entire function $F$, we use notations
$M_F(r)=\max_{|z|=r} |F(z)|$,
$m_F(r)=\min_{|z|=r} |F(z)|$,
 and put
$n_F(r)$ to be the number of zeros of $F$ in the disc $|z|<r$,
counted with multiplicities. In
  %\margp{\small Thm. A slabee \cite[Theorem~3.6.1]{boas-book}. Vozmozhno, verna dlya bolee shirokogo klassa funksij.}
what follows, we will say that an
entire function $F$ \textit{is of class $\slw$} if it is of zero
order and $\log M_F(r)= \cO\big((\log r)^2\big)$ as $r\to\infty$;
the last condition can be replaced by the condition
$n_F(r)=\cO(\log r)$. We will use the following version of
\cite[Theorem~3.6.1]{boas-book}.

\begin{theoremB}
For any entire function $F$ of the class $\slw$, which
is not a polynomial, and
any $N\in\BN$, one has $\limst_{z\to \infty} |z|^{-N} |F(z)| = +\infty$.
\end{theoremB}

\begin{lemma}
\label{lemma-entire}
Let $x\in H$ and $y\in M^\perp$. Then the entire function
$f_{x,y}$ belongs to the class $\slw$.
\end{lemma}

\begin{proof}[Proof]
It is well-known that $\LL^*|M^\perp$ is quasinilpotent
(see, for instance, \cite{markus-book}).
Given a linear operator $B$ on a Hilbert space,
we denote by $\{\mu_j(B)\}_{j\ge 1}$ the sequence
of its singular numbers.
The lacunarity of the spectrum of $\A$,
the property that $\rank P_n\le 2$ for all $n$ and
the well-known
estimates for singular numbers of the sum of two operators
\cite[Corollary~XI.9.3]{Dunf_Schwarz} imply
that
\[
\mu_j(\LL^*|M^\perp)\le \mu_j(\LL^*)\le C \si^j,
\]
where $\si<1$ is a constant. In particular,
$\LL^*|M^\perp$ is a trace class operator.
By applying the arguments employed in the proof of
\cite[Theorem~XI.9.26]{Dunf_Schwarz}, we get that
\begin{multline*}
\|(I-\bar\la \LL^*)^{-1}|M^\perp\| =
\|\det\big((I-\bar\la \LL^*)|M^\perp\big)\cdot(I-\bar\la \LL^*)^{-1}|M^\perp\| \\
\le \prod_{j=1}^\infty
      \mu_j\big((I-\bar \la \LL^*)|M^\perp \big)
\le \prod_{j=1}^\infty
      \big(1+|\la|\mu_j(\LL^*|M^\perp) \big)
\le \exp\big(C (\log|\la|)^2\big)
\end{multline*}
(notice that $\det\big((I-\bar\la \LL^*)|M^\perp\big)\equiv 1$ for all $\la$).
Since
\[
|f_{x,y}(\la)|= \big|\langle x, (I-\bar\la \LL^*)^{-1} y\rangle \big|
\le
\|x\|\cdot\|y\|\cdot\big\|(I-\bar\la \LL^*)^{-1}|M^\perp\big\|,
\]
the assertion of Lemma follows.
%Ner-va dlya sing. chisel -  \cite[Corollary~XI.9.3]{Dunf_Schwarz}.
%\margp{Dopishu pozzhe - D.V.}
\end{proof}

\begin{lemma}
\label{lemma-mer-lac}
For any normal compact operator $\A$ with lacunary spectrum and any
$\de >1$, one has
$\limst_{z\to\infty} |z|^{-\de}\, \|(\A-z^{-1})^{-1}\| = 0$.
\end{lemma}

\begin{proof}
Let $1<\de_1<\de$, and consider the discs $B_n:=B(s_n,
|s_n|^{\de_1})$.
Assume that $|z|\ge 1$. If $z^{-1}\notin\cup_n B_n$, then
$|z^{-1}-s_n|\ge \eps |z|^{-\de_1}$ for all $n$, where $\eps>0$ is
some constant (consider the cases $|z^{-1}| \ge 2|s_n|$ and
$|z^{-1}| < 2 |s_n|$). Therefore $|z|^{-\de_1}
\|(z^{-1}-\A)^{-1}\| \le \eps^{-1}$ for all $z$
such that $z^{-1}  \notin \BC\sm \cup_n B_n$.
One gets from the lacunarity of the
spectrum that the set $ \{|z|^{-1}: \; z\in \cup_n B_n\} $ has
linear density zero, which implies the statement.
\end{proof}

\begin{lemma}
\label{lemma-about-beta}
If $\LL$ is not nearly complete, then
for any $s\in \BN$,
\[
\limst_{z\to \infty} z^s \beta(z)=0,
\]
where $\be(z)$ is defined in \eqref{beta}.
\end{lemma}

\begin{proof}[Proof]
Suppose $\LL$ is not nearly complete, so that
$M^\perp$ is infinite dimensional.
Fix some $s\in \BN$. By
Proposition~\ref{prop-nrl-compl},
Lemma~\ref{lemma-entire} and Theorem~B,
there exist $x\in H$ and $y\in M^\perp$ such that
$\limst_{z\to \infty} |z|^{-s-3} |f_{x,y}(z)| = +\infty$.
By Lemma~\ref{lemma-mer-lac},
$$
\limst_{z\to \infty} z^{-2} u^*(\A-z^{-1})^{-1}v = 0
$$
for any pair of vectors $u,v\in H$.
Since the limit $\limst$ of the modulus of the left hand part in \eqref{id}
equals $+\infty$ and the finite intersection of any subsets of
$[0, +\infty)$ of linear density one
has linear density one, the assertion of the lemma follows.
\end{proof}
\medskip

\section{Proof of Theorems \ref{comp-var-thm1} and \ref{3to1}}
\label{mainpr}

First we show how to deduce Theorem \ref{3to1} from Theorem \ref{comp-var-thm1}.

\begin{proof}[Proof of Theorem \ref{3to1} assuming Theorem \ref{comp-var-thm1}]
Let $\LL$ have the form given in the Theorem.
Put $t_n=s_n^{-1}$
 and define $\be(z)$ by
\eqref{beta}, where $c_n= - t_n^2 \langle P_n a, b\rangle$.
Assume $\LL$ is not nearly complete; then by Lemma~\ref{lemma-about-beta},
equality~\eqref{limst2}
holds for any positive integer $s$. Therefore
Theorem~\ref{comp-var-thm1} gives us the conclusions
of Theorem~\ref{3to1}.
\end{proof}

The rest of this section is devoted to the proof of
Theorem \ref{comp-var-thm1}.

\begin{lemma}
\label{div-diff} Let $r$ be a positive integer and let $f\in C^r[a,b]$ be a real function. If
$|f^{(r)}|>\eps>0$ on $[a,b]$, then there exists
a subinterval $[c,d]$ of $[a,b]$ of length $ \frac {b-a}{3^r}$
such that $|f(x)| \ge \big(\frac {b-a} 6\big)^r\,\eps$ for all $x\in [c,d]$.
\end{lemma}

\begin{proof}[Proof]
Consider first the case when $r=1$. Then we can assume without loss of generality that $f'>\eps$ on $[a,b]$.
Let $[a,b]=I_1\cup I_2 \cup I_3$ be the subdivision of
$[a,b]$ into three subsequent equal intervals. Then one can take
$[c,d]=I_1$ if $f(\frac {a+b}2)<0$ and $[c,d]=I_3$ in the opposite case.

The case of general $r$ now follows by an obvious induction argument.
\end{proof}

We will use the following result by G. P\'olya (1923).

\begin{lemmaC}[see P\'olya \cite{Polya23}, p. 170]
\label{3seqs}
Let $\{p(n)\}$, $\{\alpha(n)\}$ and $\{q(n)\}$ $(n\in \BN$) be sequences
such that $p(n)\ge0$, $q(n)\ge0$, $\alpha(n) > 0$, and
$q(n)=\alpha(n)p(n)$ for all $n$. Suppose that
$\limsup p(n)=+\infty$,
$\lim q(n)=0$, and that the sequence $\{\alpha(n)\}$ decreases and tends to $0$.
Then there exists an increasing index sequence $\{m_k\}$ such that
\begin{enumerate}
\item $p(m_k)=\max \big\{ p(s): \; 1\le s\le m_k \big\}$ for all $k$;

\item $q(m_k)=\max \big\{ q(s): \; s\ge m_k \big\}$ for all $k$;

\item $\lim_k p(m_k)=+\infty$.
\end{enumerate}
\end{lemmaC}

The main step in the proof of
Theorem \ref{comp-var-thm1} will be the following statement.

\begin{lemma}
\label{comp-var-lem-old}
Suppose that the sequences $\{t_n\}$,
$\{c_n\}$ and a complex number $\varkappa$ meet all the conditions of
the above Theorem~\ref{comp-var-thm1}, but instead
of~\eqref{limst2}, we only require that \beqn \label{limst3}
\limst_{|z|\to +\infty}  z \be(z) =0. \neqn Then $\sum_n
|t_n^{-1}\,c_n| < + \infty$ and $\sum_n t_n^{-1}\,c_n = \varkappa$.
\end{lemma}

\begin{proof}
Since $\{t_n\}$ is lacunary, it follows that for some constant $\gamma>0$,
$|t_m-t_n|\ge \gamma\max(|t_m|,|t_n|)$ for all $m\ne n$.
Also, there are some constants $g, B>1$ such that $|t_n/t_m|\le B g^{n-m}$ for all $n<m$.
We assume that the sequence $\{|t_n|\}$ increases.

First let us prove that $\sum_n |t_n^{-1}\,c_n|<\infty$.
Assume it is not so.
Choose some $u\in (1,g)$ close to $1$.
Put
$p(n)=u^n |c_n| |t_n|^{-1}$,
$\alpha(n)=u^{-n}|t_n|^{-1}$,
$q(n)= |t_n|^{-2} |c_n|$.
Since the series $\sum_n |t_n^{-1} \, c_n|$
diverges and the series $\sum_n |t_n^{-2} \, c_n|$
converges, it follows that all the hypotheses of P\'olya's lemma
are satisfied. Let $\{m_k\}$ be an index sequence that has properties (1)--(3).

Properties (1) and (2) imply that
\beqn
\label{est-c-n}
\begin{alignedat}{2}
|c_n|& \le u^{m_k-n} \frac {|t_n|}{ |t_{m_k}| }\, |c_{m_k}| \le B (g^{-1} u)^{m_k-n}\, |c_{m_k}| & \quad & \text{for}\; n<m_k; \\
\frac{|c_n|}{|t_n|^2} & \le \frac{|c_{m_k}|}{|t_{m_k}|^2}\,  &\quad & \text{for}\; n>m_k.
\end{alignedat}
\neqn
By \eqref{beta2},
$$
\frac {\beta''(z)} {2}\,
=
\sum_{n=1}^\infty \frac {c_n}{(z - t_n)^{3}}.
$$
Let $|z-t_{m_k}|\le \eps |t_{m_k}|$, where
$\eps$ is a small positive constant, which will be chosen later.
If $n<m_k$, then
\[
|z-t_n|
\ge
|t_n-t_{m_k}| - |z-t_{m_k}|
\ge
(\gamma-\eps)|t_{m_k}|.
\]
Similarly, if $n>m_k$, then
$
|z-t_n|
\ge
|t_n-t_{m_k}| - |z-t_{m_k}|
\ge
(\gamma-\eps)|t_n|.
$
Hence inequalities \eqref{est-c-n} imply the estimates
\beqn
\label{est-der}
\begin{aligned}
& \Big|  \frac {\beta''(z)} 2
     -  \frac {c_{m_k}}{(z-t_{m_k})^{3}}\Big| \\
     & \quad \le
\sum_{n=1}^{m_k-1} \frac {|c_n|}{    (\gamma-\eps)^{3} \, |t_{m_k}|^{3}}
+ \sum_{n=m_k+1}^\infty \frac {|c_n|}{(\gamma-\eps)^{3}\, |t_n|^{3} } \\
& \quad \le
 \frac {B|c_{m_k}|}{|t_{m_k}|^{3} (\gamma-\eps)^{3} } \,
       \sum_{n=1}^{m_k-1} \big( \frac u g \big)^{m_k-n}
+ \frac B {(\gamma-\eps)^{3} }
\sum_{n=m_k+1}^\infty \frac {|c_{m_k}|} {|t_{m_k}|^3}\; g^{-n+m_k} \\
& \quad \le
K(\eps)\,
\big|\frac {c_{m_k}}{(z-t_{m_k})^{3}}\big|
\qquad\quad \text{for} \enspace
0<|z-t_{m_k}|<\eps|t_{m_k}|,
\end{aligned}
\neqn
where
\[
K(\eps) \defin \, B \Big( \frac\eps{\gamma-\eps} \Big)^{3} \Big[
\frac {u}{g-u} + \frac 1 {g-1} \Big].
\]
Choose a small $\eps\in (0, \gamma)$
such that $K(\eps)<\frac 1 2$.
Assume now that $z\in [(1+\eps_1) t_{m_k}, (1+\eps) t_{m_k})$,
where $\eps_1\in (0,\eps)$ is a constant.
Then the inequality
$|a^{-3}-b^{-3}|\le 3|a-b|\big( \min(|a|, |b|) \big)^{-4}$ and
\eqref{est-der}, together with the triangle inequality, give
\beqn
\label{incircle}
\Big|\frac {\beta''(z)} 2
     -  \frac {c_{m_k}}{((1+\eps_1)t_{m_k}-t_{m_k})^{3}}\Big|
     \le
\frac 23 \; \Big| \frac
{c_{m_k}}{((1+\eps_1)t_{m_k}-t_{m_k})^{3}}\Big| \neqn if $\eps_1$
is sufficiently close to $\eps$. By property (3) from Lemma C,
$|c_{m_k}|=u^{-m_k}t_{m_k} p(m_k)\ge \eps_2>0$ for all $k$. Now it
follows from \eqref{incircle} that there are constants $\zeta\in
\BC$, $|\zeta|=1$ and $\rho>0$ such that
$$
|f''(t)| \ge \rho {|t_{m_k}|^{-3}} {|c_{m_k}|}
\ge \rho \,\eps_2\, |t_{m_k}|^{-3}
$$
for all $k$ and all $t \in [(1+\eps_1) |t_{m_k}|, (1+\eps)|t_{m_k}|)$, where
$f(t)=\Re \big(\zeta \beta( t\cdot t_{m_k}/|t_{m_k}|)\big)$.
So Lemma \ref{div-diff} yields that there
is a subinterval of $[(1+\eps_1) |t_{m_k}|, (1+\eps)|t_{m_k}|)$
of length $(\eps-\eps_1) |t_{m_k}|/9$ on which
$|f(t)|\ge \eps_3 t^{-1}$, where $\eps, \eps_1, \eps_3>0$ do not depend on $k$.
This contradicts the assumption \eqref{limst3}.

We conclude that the sum $\sum_n |t_n^{-1}\,c_n|$ converges.
Take some $\tau\in (0, \frac \gamma 2)$. Then
the discs $B(t_n, \tau\, |t_n|)$
are pairwise disjoint. Let $U=U(\tau)$ be their union.
Choose $\tau$ so small that the set
\[
A=\{r>0: \partial B(0,r)\subset \BC\sm U(\tau)\}
\]
satisfies $\limsup_{R\to\infty} R^{-1}\,\big|A\cap [0,R]\big|>0$.
Now one can apply the Lebesgue dominated convergence theorem to
the sum $\be(z) = \varkappa + \sum_n \frac{c_n z}{t_n(t_n-z)}$. Since
$|z/(t_n-z)|\le C(\tau) <\infty$ for all $z\notin U$ and all $n$,
one gets that the limit of $\beta(z)$ as
$|z|\to\infty, |z|\in A$ exists and equals $\varkappa - \sum_n t_n^{-1} c_n$.
Hence $\sum_n t_n^{-1}\,c_n = \varkappa$.
\end{proof}

\begin{proof}[Proof of Theorem \ref{comp-var-thm1}]
Given an integer $\ell\ge 1$, consider the statements:
\vskip.2cm

$(1)_\ell$ \enspace $\sum_n |c_n t_n^{\ell-2}|<\infty$;

\vskip.2cm

$(2)_\ell$ \enspace $\sum_n c_n t_n^{\ell-2}=0$ if $\ell\ge 2$
and $\sum_n c_n t_n^{\ell-2}=\varkappa$ if $\ell=1$;

\vskip.2cm

$(3)_\ell$ \enspace $\beta(z)=z^{-\ell} \beta_{\ell}(z)$, where
\[
\beta_\ell(z)= \sum_n c_n t_n^\ell \bigg(\frac 1 {t_n-z} - \frac 1 {t_n}\bigg)
= z \sum_n c_n t_n^{\ell-1} \frac 1 {t_n-z}. % , \qquad m\ge 1.
\]

\vskip.2cm

\noindent Lemma \ref{comp-var-lem-old} gives $(1)_1$, $(2)_1$ and
$(3)_1$. If for some $\ell\ge 1$, $(1)_\ell$, $(2)_\ell$ and
$(3)_\ell$ have been obtained, one applies Lemma
\ref{comp-var-lem-old} to $\wt c_n=c_n t_n^\ell$ and to
$\beta_\ell$ and gets properties $(1)_{\ell+1}$ and
$(2)_{\ell+1}$, which imply that $\beta_{\ell+1}(z)= z
\beta_\ell(z)$. This gives $(3)_{\ell+1}$.

So, by induction, the equalities $(1)_\ell$, $(2)_\ell$ and $(3)_\ell$ hold for any $\ell\ge 1$.
\end{proof}
\medskip

\section{Some additional remarks}
\label{addit}

First let us give an example of a natural class of compact normal lacunary
operators.

\begin{example}
\label{ex-lac-comp}
Let $\eta$ be an analytic function in an annulus $\an=\{z\in \BC: r<|z|<R\}$, where
$0<r<1<R$. Suppose it has the form $\eta(z)=\eta_0(z)+\eta_1(z)$, where
$\eta_0(z)=\tau_1 (R-z)^{-a} +  \tau_2 (r^{-1}-z^{-1})^{-a}$ and
$\eta_1$ is smoother than $\eta_0$ in the sense that
$\eta_1\in H^p(\an)$ for some $p>1/a$
(the powers are defined by using the principal branch of the logarithm).
Assume that $a>1$,
$\tau_1, \tau_2\in \BC$ are nonzero and
$\tau_1/\tau_2\notin (0,+\infty)$. Denote by $\BT$ the unit circle in $\BC$.
Then the convolution operator
\[
(\A_\eta f)(e^{i\tht}) = \frac 1{2\pi}\,\int_0^{2\pi} \eta(e^{i(\tht-x)})f(e^{ix})\, dx
\]
on $L^2(\BT)$
is compact, normal and has lacunary spectrum. Indeed,
in the Fourier representation $\A_\eta$ is just the multiplication operator
on $\ell^2(\BZ)$
by the sequence of Fourier coefficients $\{\hat\eta(n)\}$ of the function
$\eta|_{\BT}$.
Notice that $\hat\eta_0(n)\sim \tau_1 R^{-a-n} n^{a-1}$ as
$n\to+\infty$ and $\hat\eta_0(n)\sim \tau_2 r^{a - n} |n|^{a-1}$ as
$n\to -\infty$, whereas the assumption on $\eta_1$ implies that
$\hat\eta_1(n)/\hat\eta_0(n)\to 0$ as $n\to \pm\infty$, see
\cite[Theorem 6.4]{Dur}.
This gives our assertion.
\end{example}

The next proposition gives an explicit form
of the space $\cup_n \ker \LL^{*n}$ whenever this space is finite dimensional.

\begin{proposition}
\label{prop-LL}
Let $\LL$ be given by \eqref{LL} and not all conditions
$({\rm M }_{n})$ are fulfilled. Choose the integer
$k\ge 0$ so that
$({\rm M }_{1}), \dots, ({\rm M }_{k})$ are fulfilled
and $({\rm M }_{k+1})$ either is not fulfilled or has no sense (that is,
the sum diverges). Let $\ell\ge 0$ be the largest integer such that
$b\in \Ran \A^{*\ell}$, and put $b=\A^{*s} b_s$, where $b_s\in H$
($s=1, \dots, \ell$).  Then
\[
\cup_n \ker \LL^{*n}= \ker \LL^{*m} = \spa \{b_1, \dots, b_m\},
\]
where $m=\min(k, \ell)$.
\end{proposition}

We omit the proof, which is completely straightforward.
The same calculations imply Proposition~\ref{prop1}.

Let $\LL_1, \LL_2$ be two bounded operators
on Hilbert spaces $H_1$, $H_2$, respectively.
In \cite{baryak}, we used the following definition:
$\LL_2$ is said to be \textit{$d$-subordinate} to $\LL _1$
 ($\LL _1\precd \LL _2$) if there exists a bounded linear operator
$Y:H_1\to H_2$, which intertwines $\LL _1$ with $\LL _2$ and has a dense range:
\[
Y\LL _1=\LL _2Y; \qquad \clos\Ran Y=H_2.
\]

As it was mentioned there, if
$\LL _1\precd \LL _2$ and $\LL _1$ is complete then
$\LL _2$ is complete. In connection with the present article, we can also mention the following fact.
\begin{proposition}
If
$\LL_1$ and $\LL_2$ are compact,
$M(\LL_1)=\clos \LL_1^N H$ and
$\LL_2$ is $d$-subordinate to $\LL _1$
then $M(\LL_2)=\clos \LL_2^N H$.
\end{proposition}

The proof is straightforward, and we leave it to an interested reader.

As it follows from \eqref{reslv-L-A} and Theorem~\ref{comp-var-thm1},
if some of the infinite sequence of moment equalities
\eqref{seqs-eqs} fail, then there is an estimate
\[
\|(\LL-z)^{-1}\|\le |z|^{-s}
\]
for the resolvent of $\LL$ for a set of the form $\{z:\;
|z|^{-1}\in A\}$, where $A\subset [0,+\infty)$ is a closed subset
of linear density $1$ at infinity. This can be compared with the estimates, which are used in the proof of
the Keldy\v s theorem: in the conditions of this theorem,
for any $\eps>0$, an estimate
$
\|(\LL-z)^{-1}\|\le C_\eps |z|^{-1}
$
holds for sufficiently small $|z|$ in the complement of the union of
the angles $\al_k-\eps\le \arg z \le \al_k +\eps$
(see \cite{markus-book}, Lemma~3.2).
In particular, for each $\eps>0$ there exists $\de>0$ such that all non-zero
spectrum of $\LL$, which is contained in the disc $|z|<\de$,
lies in the union of the above angles.

The next example shows
that for weak perturbations $\LL=\A(I+S)$ that do not satisfy
the requirement $\ker(I+S)=0$,
the last geometric property of the spectrum does not hold in general.

\begin{example}
\label{ex}
There is an operator
$\LL=\A(I+S)$, which has the form \eqref{LL}, where $\A$ is a cyclic selfadjoint operator
with lacunary spectrum $\{2^{-n}: \enspace n\in \BN\}$ such that $\si(\LL)= \{0\}\cup \{i2^{-n}:n\in \BN, n\ge 2\}$.
\end{example}

To construct this operator, consider the function
\[
\psi(z)= \frac {2}{2-z}\, \prod_{n=2}^\infty \phi_n(z), \qquad \text{where}\quad \phi_n(z)=\frac {2^n+iz}{2^n-z}\, .
\]
Notice that, given any constant $K\in (0,1)$, there exists $C_K>0$ such that
$|\phi_n(z)-1|\le C_K 2^{-n}|z|$ if $2^{-n}|z|<K$ and
$|\phi_n(z)+i|\le C_K 2^{n}|z|^{-1}$ if $2^{n}|z|^{-1}<K$. It follows that the above product converges
for any $z\ne 2^n, n\in \BN$ and defines a meromorphic function on $\BC$. The residues
$c_n= -\operatorname{Res}_{2^n}  \psi$ satisfy $|c_n|\asymp 1$. The above estimates
for $\phi_n(z)$ imply that $\max_{|z|=3\cdot 2^k}|\psi(z)|\le C 2^{-k}$.
Put $t_n=2^n$. By writing down
the residue theorem for the function $\frac {\psi(\cdot)}{(\cdot)-\zeta}$ on the contours
$|z|=3\cdot 2^k$, where $\zeta$ is fixed and letting $k\to\infty$, one gets
\[
\psi(\zeta)=\sum_{n\in \BN} \frac{c_n}{t_n-\zeta}= 1 + \sum_{n\in
\BN} c_n\Big(\frac{1}{t_n-\zeta}-\frac{1}{t_n}\Big)
\]
(here we use that $\psi(0)=1$). Take any sequences $a=\{a_n\}$ and
$b=\{b_n\}$ such that $\{2^n a_n\}$ and $\{b_n\}$ are in $\ell^2$
and $c_n=-t_n^2 a_n b_n$. (For instance, one can put $a_n=t_n^{-3/2}$ and
$b_n=-c_n t_n^{-1/2}$.)
The operator $\A$ on $H=\ell^2$, defined by
$\A\,\{x_n\}=\{2^{-n}x_n\}$, is cyclic, compact, selfadjoint and
has trivial kernel. Since $\{2^n a_n\}$ is in $\ell^2$, the
operator $\LL=\A+ab^*$ on $\ell^2$ has the form $\LL=\A(I+S)$,
where $S$ is a rank one operator, so $\LL$ is a weak perturbation
of $\A$. Since for this perturbation, $\be(z^{-1})=\psi(z)$, we
get that the spectrum of $\LL$ is $\{0\}\cup \{i2^{-n}:n\in \BN,
n\ge 2\}$.

Notice that in this example, we get that
$\sum a_nb_n s_n^{-1} = -\sum c_n t_n^{-1}=-1$, so that
the first moment equation $(M_1)$ holds and one has
absolute convergence in its left hand part.
The general term of the sum
$\sum a_nb_n s_n^{-2}$ in $(M_2)$ does not tend to zero.
So the hypotheses of Theorem~\ref{thm0angles} do not hold.
This example shows that in this case, although the
perturbation is weak, it spectrum is not contained in the union
of angles, given by Theorem~\ref{thm0angles}.
By our Theorem~\ref{3to1}, $\LL$ and $\LL^*$ are nearly
complete.

In fact, it is easy to get that $|\beta(z)|\ge C|z|^{-1}$ for $|z|\ge 1$,
$z$ outside arbitrarily small angles around the two coordinate axes. So
in this particular case, it is easy to get that $\LL$ and $\LL^*$ are nearly complete either
by using the argument of the proof of Theorem~\ref{thm0angles} or the criterion of completeness, given in
\cite{baryak}, Proposition~3.1.

\section{The case of singular perturbations of unbounded normal operators}
\label{singul}

In this section we consider an analogue of Theorem~\ref{3to1} for
rank one perturbations of unbounded normal operators. Let $\A_0$
be a compact normal operator and let $\LL_0$ be its rank one
perturbation such that $\ker \A_0 = \ker \LL_0 = 0$. Put $\A =
\A_0^{-1}$ and let $\LL = \LL_0^{-1}$ be the algebraic inverse of
$\LL_0$ defined on the range of $\LL_0$ (In the last two sections,
to distinguish between the bounded and unbounded operators, we use
the notation $\A_0$, $\LL_0$ for compact operators and $\A$, $\LL$
for their unbounded inverses.) One should expect that $\LL$ is in
a certain sense a rank one perturbation of the unbounded normal
operator $\A$. However, it is not necessarily a relatively compact
(rank one) perturbation $\A$. To formalize this we need to
introduce the notion of a singular rank one perturbation.

Now let $\A$ be an unbounded normal operator on a Hilbert space
$H$ and let $G(\A)$ stand
for the graph of $\A$; it is a subspace of $H\oplus H$.
We assume that $\A^{-1}$ exists and is bounded.
We say that $\LL$ is a \textit{singular balanced rank one perturbation of $\A$}
if $G(\A)\cap G(\LL)$ has codimension one in both spaces $G(\A)$ and
$G(\LL)$. Here we follow~\cite{baryak}; in this paper
a definition of singular balanced rank $n$ perturbations of a
not necessarily normal operator $\A$ was also given. If
one takes for $\A$ an ordinary differential operator on an
interval and changes its defining boundary conditions without
changing the formal differential expression, then one obtains this
kind of perturbation of $\A$.

If $\LL_0$ is a usual rank one perturbation of $\A_0 = \A^{-1}$, which
has zero kernel and $\LL=\LL_0^{-1}$ is its algebraic inverse,
then $\LL$ is a singular
rank one perturbation of $\A$; moreover, as it will be explained
a little bit later, ``most'' of singular rank
one perturbations of $\A$ are obtained in this way. In this Section we
reformulate our completeness result, Theorem~\ref{3to1}, for
singular rank one perturbations, and give some examples where
singular rank one perturbations of operators with lacunary spectra
appear.

To give a description of all singular rank one perturbations, we
need to introduce some new notions.
We define \textit{the extrapolation Hilbert space} $\A H$ as the set of
formal expressions $\A  x$, where $x$ ranges over \textit{the
whole space} $H$. Put $\|\A x\|_{\A H}= \|x\|_H$ for all $x\in H$.
The formula $x=\A (\A^{-1}x)$ allows one to interpret $H$ as a
linear submanifold of $\A H$. We consider the scale of spaces
$$
\cDA \subseteq H \subseteq \A H.
$$
Notice that $\cDA =\cD(\A^*)$. The pairing
$\langle x, y \rangle \defin \langle \A x, \A^{*-1}y \rangle$,
$x\in \cDA, y\in \A H$, gives rise to a natural identification
$\cDA = (\A H)^*$.

The set of rank $1$ singular balanced perturbations of $\A $ can
be parametrized by what we will call $1$-data for $\A $. By
\emph{$1$-data} we mean a triple $ (a,b,\varkappa)$, where $a, b\in \A
H$ are non-zero, $\varkappa\in \BC$ and the following condition is
fulfilled:

\smallskip

{\bf (A)} If $a\in H$, then $\varkappa\ne \langle \A^{-1}a, b\rangle$.

\smallskip

Given $1$-data $(a,b,\varkappa)$, the corresponding
rank $1$ singular balanced perturbation
$\LL= \LL(a,b,\varkappa)$ of $\A$ is defined as follows:
\beqn
\begin{aligned}
\label{2app}
\cD(\LL )&\defin \big\{
y=y_0+c\A^{-1}a: \\
& \qquad \qquad c\in \BC,\, y_0\in \cDA,\, \varkappa  c+b^*y_0=0
\big\};                     \\
\LL \, y& \defin \A y_0, \quad y\in \cDL.
\end{aligned}
\neqn Condition  $(A)$  is equivalent to the uniqueness of the
decomposition $y=y_0+c\A^{-1}a$ for $y\in \cD(\LL)$ and hence to
the correctness of the definition of $\LL$.

As it is shown in ~\cite{baryak}, any singular balanced rank one
perturbation of $\A$ has a form
$\LL = \LL(a,b,\varkappa)$ for certain $1$-data $(a,b,\varkappa)$.
Moreover, if $\varkappa\ne 0$, then, by \cite[Proposition~1.6]{baryak},
$\LL= \LL_0^{-1}$ (the algebraic inverse), where $\LL_0$ is a rank one perturbation
of $\A_0 = \A^{-1}$:
\beqn
\label{inverse}
\LL_0 = \A_0 - \varkappa^{-1} \A_0 a (\A_0 b)^*.
\neqn

One can write
\beqn
\label{unb-norm}
\A x = \sum_{n\in \BN} t_n
P_n x,
\neqn
where the finite dimensional orthogonal projections
$P_n$ are as above: $P_nP_m=0$ for $m\ne n$, $\sum_n P_n=I$, but
now $|t_n|\to \infty$
(and $t_n\ne 0$ for all $n$).
The domain of $\A$ is the set of vectors
$x\in H$, for which the above sum converges.

We will say that the singular perturbation $\LL(a, b, \varkappa)$ is \textit{degenerate}
if $\langle P_n a, b\rangle =0$ for all $n$ and at the same time $\varkappa=0$
(it is consistent with the condition (A) if $a\notin H$).
We will say that $\LL(a, b, \varkappa)$ is \textit{non-degenerate} in all other cases.

It is easy to check that the spectrum of $\LL$
coincides with its point spectrum and equals to the zero set of the meromorphic function
\[
\be_{\LL}(\la)
=\varkappa+\la b^*(\A -\la)^{-1}\A^{-1} a
= \varkappa + \la \sum_k
\,
\frac{t_n \langle P_n \A^{-1} a, \A^{-1}b \rangle} {t_n-\la}
\]
(see \cite{baryak}).
So, if the operator $\LL(a, b, \varkappa)$ is degenerate, then each point $\la\in \BC$ is its eigenvalue.
If $\LL(a, b, \varkappa)$ is non-degenerate, then its spectrum is discrete.

Whereas the point $0$ was a special point of the spectrum for compact operators,
in the present context of unbounded operators, this role passes
to the point $\infty$. Analogously to the case of bounded operators,
given an operator $\LL$ on $H$ with compact resolvent,
we say that $\LL$ is \textit{complete} if its root vectors span $H$
and we say that $\LL$ is \textit{nearly complete} if its root vectors span
a subspace of $H$ of finite codimension.

\begin{theorem}
\label{thm-sing-pert} Let
$\A$ be a normal operator with compact resolvent, given by~\eqref{unb-norm},
which has lacunary spectrum $\{t_n\}_{n\in \BN}$.
Suppose that $0\notin\si(\A)$.
Let $(a,b,\varkappa)$ be $1$-data for $\A$, and let $\LL=\LL(a,b,\varkappa)$ be the corresponding
singular perturbation of $\A$, which is non-degenerate.
If $\LL$ is not nearly complete, then the following infinite sequence
of ``moment'' equations holds for all $k\in \BZ$, $k\ge -1$:
\[
\label{seqs-eqs-sing}
({\rm S }_k) \quad \qquad\qquad\qquad\qquad
\sum_n t_n^k \langle P_n a, b\rangle =
\begin{cases}
  \varkappa, \quad & k= -1, \\
  0, \quad & k \ge 0. \\
\end{cases}
\qquad\qquad
\]
\end{theorem}

\begin{proof}[Proof]
We reduce this assertion to the case of a compact operator.
First we observe that it is suffices to consider the case
when $\varkappa\ne 0$. Indeed,
by \cite[Proposition~1.6]{baryak},
for any $\la\in \rho(\A )$, one has
\beqn
\LL (\A ,a,b,\varkappa)-\la I=\LL \big(\A -\la I, a,b,\be_{\LL}(\la)\big).
\neqn
Notice that the function $\be_{\LL}(\la)$ has poles exactly at
the points $\{t_n\}$ and
its residue at $t_n$ equals
$t_n^2\langle P_n \A^{-1} a, \A^{-1}b \rangle$.
If all these residues are zero for all $n$, then
$\varkappa\ne 0$, by the non-degeneracy assumption.

If not all numbers $\langle P_n a, b \rangle$ are zero,
then $\be_{\LL}$ is non-constant, and therefore
$\be_{\LL}(\la)\ne0$ for some $\la$.
Operator $\LL-\la I$ is nearly complete if and only if $\LL$ is.
A direct calculation shows that moment equations
$({\rm S }_k)$ hold for $\LL$ if and only if they hold
for $\LL-\la I$.
So
we may assume that $\varkappa\ne 0$, just by replacing $\LL$ with
$\LL-\la I$, where
$\la\in \BC\setminus\sigma(\A)$
is any number such that $\beta_{\LL}(\la)\ne 0$.

If $\varkappa\ne 0$, then $\LL= \LL_0^{-1}$ where $\LL_0$ is a rank one perturbation
of $\A_0 = \A^{-1}$ given by formula \eqref{inverse}.
Notice that $\A^{-1}$ is a compact normal operator with lacunary spectrum
and that $\LL$ is nearly complete if and only if $\LL_0$ is. Now the statement follows
by applying Theorem~\ref{3to1} to $\LL_0$.
\end{proof}

Next we give examples of selfadjoint operators with compact resolvents and their singular
perturbations, which are motivated by applications.

\begin{example}
Consider an infinite symmetric Jacobi matrix with
diagonal entries $b_n$ and off-diagonal entries $a_n$, $n= 0, 1, \dots$,
where at least one of the sequences
$\{a_n\}$, $\{b_n\}$ grows exponentially.
In many cases, these matrices give rise to unbounded selfadjoint operators with lacunary spectrum.
As a particular example, take the position operator of the
Biedenharn--Macfarlane $q$-oscillator, studied
by Klimyk in \cite{Klimyk2005}. It corresponds to the values $b_n=0$, $a_n=((q^n-1)/(q-1))^{1/2}$,
where $q>1$ is a fixed number.
This Jacobi matrix defines a symmetric operator $\A_0$ with indices $(1,1)$.
In \cite[Theorem~1]{Klimyk2005}, a parametrization of all self-adjoint extensions
of $\A_0$ is given and their spectra are calculated.
Any self-adjoint extension of $\A_0$ is lacunary.
Therefore Theorem~\ref{thm-sing-pert} applies to any operator
$\LL$ such that $\A_0\subset \LL$ and $\dim \cD(\LL)/\cD(\A_0)=1$.
There are many other Jacobi matrices, which give rise to different classes of
$q$-orthogonal polynomials; see
\cite{Koek-L-Sw-book}. By inspecting the orthogonality relations one sees that
there are other cases when
the corresponding unbounded selfadjoint operators are lacunary.

An easier case is obtained if
the off-diagonal entries of the Jacobi matrix are subordinate to the diagonal ones.
Namely, assume that $b_n>0$, $a_n\in \RR$,
\[
\frac{b_{n+1}}{b_n}\to Q>1
\enspace
\text{and} \enspace
\frac{a_n}{b_n}\to 0 \quad
\enspace
\text{as} \enspace n\to+\infty.
\]
Then our Jacobi matrix defines an unbounded selfadjoint
operator $\A$ on $\ell^2(\mathbb N)$, which is semi-bounded from below and has lacunary spectrum.
To see it, consider the diagonal matrix $D$ with the sequence
$\{b_n\}$ on the diagonal. Then $D$ defines a positive unbounded selfadjoint
operator on $\ell^2(\BZ_+)$. Define $\A$ by $\A x= D x+ F x$,
$x\in \cD(\A)\defin\cD(D)$, where $F$ is the
operator given by the Jacobi matrix with the same off-diagonal entries $a_n$ and
zero entries on the diagonal. It is easy to see that $\A$ is a relatively compact
perturbation of $D$ and therefore $\A$ is selfadjoint and has discrete spectrum, see
\cite[Theorem 9.9]{Weidmann}. Consider the circles $|z-b_n|=\sigma b_n$ around eigenvalues
of $D$. If $\sigma>0$ is small, each of them contains exactly one eigenvalue of $D$. Moreover,
it is easy to see that $\|(D-z)^{-1}F\|\le 1/2$ if $z$ is outside all these circles and
is sufficiently large. It follows that for large $n$,
the eigenvalues of the operator
$D+tF$ (which has discrete spectrum) do not cross the
the circle $|z-b_n|=\sigma b_n$ when the parameter $t$ traces the interval $[0,1]$.
By \cite[Lemma 8.1]{markus-book}, there is exactly one eigenvalue of $\A=D+F$ inside each such circle.
Therefore in this case, $\A$ is lacunary.

We remark that in general, to find out whether a
selfadjoint Jacobi matrix with exponentially growing entries is selfadjoint or has
defect indices $(1,1)$, one can apply
\cite{Chihara1989}, Theorem 2 and its Corollary.
\end{example}
\bigskip

\section{Proof of Theorem \ref{hhn}}
\label{sharpr}

\subsection{Functional model for rank one perturbations}
\label{funmod}

The proof of the main part of Theorem \ref{hhn}
uses a functional model for singular rank one perturbations
of unbounded selfadjoint operators with discrete spectrum, which are essentially
the algebraic inverses to rank one perturbations of compact selfadjoint operators.
This model was introduced in \cite{baryak}. Let us briefly recall its
statement (in the generality we need here).
For the details see \cite{baryak} or \cite[Section 4]{bbby}.

We consider the following objects:
\smallskip
\begin{itemize}
\item
$\{t_n\}$ is, as above, a sequence of real points such that $|t_n| \to \infty$
as $|n|\to \infty$, and $t_n\ne 0$. We can assume without loss of generality that
$\{t_n\}$ is an increasing sequence enumerated by $\mathbb{Z}$, $\mathbb{N}$ or
$-\mathbb{N}$.
%\smallskip
%\item $\mu=\sum_n \mu_n \delta_{t_n}$, $\mu_n>0$, and $\sum_n t_n^{-1}\mu_n<\infty$.
\smallskip
\item $A$ is an entire function which is real on $\mathbb{R}$
and has simple zeros  exactly at the points $t_n$.
\item Two sequences $\{a_n\}$ and $\{b_n\}$, $b_n\ne 0$ for any $n$,
and a complex number $\varkappa \ne 0$ satisfy

(1) $\sum_n \frac{|a_n|^2+|b_n|^2}{t_n^2} <\infty$;

(2)
$
\sum_n \frac{a_n \overline b_n}{t_n} \ne \varkappa
$
in the case when $\sum_n |a_n|^2 <\infty$.

\smallskip
\item
The entire function $E$ is given by $E = A-iB$, where
\begin{equation}
\label{form-e}
\frac{B(z)}{A(z)} = \delta+ \sum_n \bigg(\frac{1}{t_n-z}
- \frac{1}{t_n}\bigg) |b_n|^2,
\end{equation}
and $\delta$ is an arbitrary real constant.
It is a Hermite--Biehler function, which means that
$|E(z)|>|E(\bar z)|$ if $\Im z>0$.

\item
The de Branges space $\he$, associated with the
function $E$. It can be defined as the space
of exactly those entire functions $F$ which have the representation
$$
\frac{F(z)}{A(z)} = \sum_n \frac{u_n |b_n|}{t_n -z}
$$
for some sequence $\{u_n\} \in \ell^2$. It is a Hilbert space with the norm given by $\|F\|_{\mathcal{H}(E)} =
\|\{u_n\}\|_{\ell^2}$,
so that $\mathcal{H}(E)$ is essentially the space of the
discrete Cauchy transforms.
We refer to \cite{br} (or \cite{baryak, bbby})
for an alternative (more standard) definition of $\he$ and more background.

\smallskip
\item
Entire function $G$ is given by
\begin{equation}
\label{form-g}
\frac{G(z)}{A(z)} = \varkappa +\sum_n\bigg(\frac{1}{t_n -z} -\frac{1}{t_n}\bigg)
a_n\overline{b}_n.
\end{equation}

\item
The model operator $\T$, defined by the formulas
$$
\cD(\T) := \{F \in \he: \text{there exists}\ c=c(F)\in \BC:
zF-cG\in \he\},
$$
$$
\T F := zF - c G, \qquad F\in \cD(\T).
$$
\end{itemize}

Now the functional model from \cite[Theorem 4.4]{baryak}
combined with \cite[Proposition 2.4]{baryak}
(see also \cite[Section 4]{bbby}) can be stated as follows.

\begin{theoremD}
Any singular rank one perturbation $\LL=\LL(\A, a, b, \varkappa)$
of the selfadjoint operator $\A$
\textup(with simple spectrum and trivial kernel\textup) is unitary equivalent to
the model operator $\T$ whose parameters $\{t_n\} =\{s_n^{-1}\}$,
$A$ and $G$ are related to $a, b, \varkappa$ as above.
Conversely, any function $G$ as above appears in the model
of some rank one perturbation of $\A$.
\end{theoremD}

The reproducing kernels of the de Branges space $\he$ are given by
\[
K_w(z)= \frac{\overline{E(w)} E(z) - {E(\bar w)} \overline{E(\bar z)}}
{2\pi i(\overline w-z)} =
\frac{\overline{A(w)} B(z) -\overline{B(w)}A(z)}{\pi(z-\overline w)}\, .
\]
One has $K_w\in \he$ and $\langle F, K_w\rangle = F(w)$ for any $w\in\BC$.

If $\sum_n |b_n|^2 =\infty$ or $\sum_n |b_n|^2 <\infty$ and
$\sum_n t_n^{-1} a_n \overline b_n \ne \varkappa$, then the adjoint operator $\LL^*$
is well-defined and also is a singular rank one perturbation of $\A$
(see \cite[Proposition 2.2]{baryak}).
Moreover, in this case the eigenvectors of $\LL^*$
are mapped by the same unitary equivalence
as in the above theorem onto the
reproducing kernels
$\{K_\lambda\}_{\lambda \in Z_G}$, see \cite[Lemma 5.4]{baryak}.
By $Z_f$ we denote the zero set of an entire function $f$. To avoid
unessential technicalities, we assume that all zeros of $G$ are simple.

\subsection{Strategy of the proof of Theorem~\ref{hhn}}
\label{strat} In view of relation~\eqref{inverse} between bounded
rank one perturbations of compact normal operators and singular
rank one perturbations, we can solve an equivalent problem for
singular rank one perturbations. Let $\A_0$ be a compact
selfadjoint operator with simple spectrum $\{s_n\}$, $s_n\ne 0$
(we can identify $\A_0$ with a diagonal operator on $\ell^2$) such
that $t_n = s_n^{-1}$ satisfy condition~\eqref{beg1}. Consider the
unbounded operator $\A=\A_0^{-1}$ and assume that we were able to
construct a singular rank one perturbation $\LL = \LL(\A, a,
b,\varkappa)$ such that $\varkappa\ne 0$, $a=\{a_n\}$, $b=\{b_n\}$
satisfy $\sum_n \frac{|a_n|^2+|b_n|^2}{t_n^2} <\infty$ and
\begin{equation}
\label{glt}
\sum_n |t_n^{-1}a_nb_n| = \infty,
\end{equation}
so that  $\LL$ is incomplete with infinite defect  (that is, the
linear span of root vectors of $\LL$ has infinite codimension in
$\ell^2$).  Consider the vectors $a^0 = \A^{-1}a = \{t_n^{-1}
a_n\}$ and $b^0 = \A^{-1}b = \{t_n^{-1} b_n\}$ from $\ell^2$.
Then, by \eqref{inverse}
$$
\LL_0 = \LL^{-1} = \A_0 - \varkappa^{-1} a^0 (b^0)^*
$$
is a bounded rank one perturbation of $\A_0$ whose set of root vectors
has infinite codimension. Note that equality~\eqref{glt} coincides
with $\sum_n |s_n^{-1} a_n^0 b_n^0| = \infty$, i.e., nonexistence of the
first moment for $\LL_0$.

Thus, the statement of Theorem~\ref{hhn} is reduced to an equivalent
problem for singular rank one perturbations. In view of the above functional model,
this problem is equivalent to a completeness problem for a system
of reproducing kernels in de Branges spaces.
Namely, to prove Theorem~\ref{hhn}
we will
need
to construct an entire function $G$ such that
\begin{equation}
\label{glav}
\frac{G(z)}{A(z)} = \varkappa +\sum_n c_n \bigg(\frac{1}{t_n -z} -\frac{1}{t_n}\bigg),
\qquad
\sum_{n} \frac{|c_n|}{t_n^2} <\infty, \qquad
\sum_n \frac{|c_n|}{|t_n|}= \infty,
\end{equation}
and $\varkappa \ne 0$, but
the system $\{K_\lambda\}_{\lambda \in Z_G}$ has an infinite defect
in $\he$. Then we can define a singular rank one perturbation $\LL$
of $\A$ such that $\LL^*$ is also a well-defined singular rank one perturbation,
but $\LL^*$ will not be complete.

We will construct $G$ of the form $A/S$ where $S$ is a canonical product
of order less than one, whose zeros
form a subset of the set $\{t_n\}$.
Then, necessarily, $\varkappa = 1/S(0)$, $c_n = - 1/S'(t_n)$
and the equation~\eqref{glav} rewrites as
\begin{equation}
\label{inter}
\frac{G(z)}{A(z)} = \frac{1}{S(z)}= \frac{1}{S(0)}
-\sum_{t_n \in Z_S} \frac{1}{S'(t_n)} \bigg(\frac{1}{t_n -z} -\frac{1}{t_n}\bigg).
\end{equation}

In Subsection \ref{prelim1} the relations between conditions
\eqref{beg1} and \eqref{beglog2} will be discussed as well as some equivalent
forms of condition \eqref{beg1}. The function $S$ from \eqref{inter}
will be constructed in Subsection~\ref{keyl}, while in Subsection~\ref{eqfor}
the proof of Theorem~\ref{hhn} will be completed.

\subsection{Discussion of condition~\eqref{beg1}}
\label{prelim1}
We begin with the proof of the fact that
any separated sequence $T=\{t_n\}$, for which \eqref{beg1} does not hold,
satisfies $n_T(r) = O(\log^2 r)$ (and thus is sufficiently sparse).

\begin{lemma}
\label{bon}
Let $R>0$ and let the interval $[R, 2R]$ contain at least $2M$ points $t_k$. Then
there exists $t_n \in [R, 2R]$ such that
$$
\prod_{k:\, R \le t_k \le 2R,\,  k\ne n}
\bigg|\frac{t_k-t_{n}}{t_k}\bigg| \le 2^{-M+1}.
$$
\end{lemma}

\begin{proof}
Clearly, we can choose $M$ points $t_{n_1}, \dots t_{n_M} \in \{t_k\}$
such that $|t_{n_1} - t_{n_j}| \le R/2$, $j=2, \dots, M$.
Hence, we have
$$
\prod_{k:\, R \le t_k \le 2R,\, k\ne n_1}
\bigg|\frac{t_k-t_{n_1}}{t_k}\bigg| \le
\bigg(\frac{R}{2}\bigg)^{M-1} \bigg(\frac{1}{R}\bigg)^{M-1} = 2^{-M+1}
$$
(we dropped the factors for which
 $|t_k -t_{n_1}|>R/2$ since they
are anyway smaller than~1).
\end{proof}

\begin{corollary}
\label{4to5}
Condition~\eqref{beglog2} implies~\eqref{beg1}.
\end{corollary}

\begin{proof}
Assume that \eqref{beg1} is not satisfied.
Given $R>0$, put
$M=M(R)=[M'(R)/2]$, where $M'(R)$ is the number of points $t_k$ in
the interval $[R,2R]$. Then, by Lemma \ref{bon}, $R^N
2^{-M(R)}\gtrsim 1$ for some $N$ which is independent on $R$.
Hence $M(R) = O(\log R)$. Thus, we conclude that, for any $m\in
\NN$, there is always no more than $O(m)$ points
 $t_n$ between
$2^m$ and $2^{m+1}$, and so $n_T(r) = O(\log^2 r)$,
which contradicts \eqref{beglog2}.
\end{proof}

\begin{remark_}
{\rm
One can rewrite \eqref{beg1} in equivalent ways.
Recall that \textit{the Krein class} consists
% \marginpar{Dat' pozzhe opredelenie?}
of entire functions $F$ which are real on
$\BR$, have only simple and real zeros and for some positive
integer $k$ and some polynomial $P$ of degree at most $k$ satisfy the
following absolutely convergent expansion:
$$
\frac{1}{F(z)} = P(z) + \sum_n \frac{1}{F'(t_n)}
\bigg(\frac{1}{z-t_n} +\frac{1}{t_n} +\dots +
\frac{z^{k-1}}{t_n^k} \bigg),
$$
where $s_n$ are zeros of $F$.

It is not difficult to see that condition \eqref{beg1} is equivalent
to any of the following properties $(i)$ and $(ii)$:
\smallskip

$(i)$ There exists an entire function $Q$ of order less than one,
whose zeros are simple and lie in the set $\{t_n\}$, such that for any $N$ we
have $\liminf_{t_n\to\infty, t_n\in Z_Q} |t_n^N Q'(t_n)| =0$
(the lower limit is taken here over those $t_n$ which are zeros of $Q$).
\smallskip

$(ii)$ There exists a subsequence of $\{t_n\}$, which is not the zero set
of a function in the Krein class.
\smallskip

\noindent For instance, to prove that \eqref{beg1} implies $(i)$, it suffices to put
\[
Q(z) = \prod_k\; \prod_{m:\; t_{n_k}/2 \le t_m\le t_{n_k}} \bigg(1-\frac z {t_m}\bigg),
\]
where the sequence $\{n_k\}$ grows fast enough
(by the above proof of Lemma~\ref{hon},
this product defines a zero order entire function).
We leave the details to the reader.
}
\end{remark_}

\subsection{A key lemma}
\label{keyl}
To construct a function $S$ satisfying
\eqref{inter} we will need the following lemma, which is the main
technical step in the proof of Theorem~\ref{hhn}.

\begin{lemma}
\label{hon}
Under the hypothesis of Theorem \ref{hhn}, there
exists a canonical product $S$
of order less than one with zeros in the set $\{t_n\}$ such that
$$
\sum_{t_n \in Z_S} \frac{1}{t_n^2 |S'(t_n)|}<\infty, \qquad
\sum_{t_n \in Z_S} \frac{1}{|t_n S'(t_n)|} =\infty.
$$
\end{lemma}

\begin{proof}[Proof]
Without loss of generality we assume that $t_n>0$
and that~\eqref{beg1} holds.
Put
$S=\prod_k S_{T_k}$,
where
\begin{equation}
\label{sk}
S_{T_k}(z) = \prod_{t_n \in T_k}\bigg( 1- \frac{z}{t_n} \bigg),
\end{equation}
and
$T_k \subset \{t_n: n\ne 1, \; t_{n_k}/2 \le t_n\le 2 t_{n_k}\}$,
where $n_k$ go rapidly to infinity
so that, for any $N>0$,
\beqn
\label{tnk-prod}
t_{n_k}^N \cdot \enspace \prod_{l \ne n_k:\,\, t_{n_k}/2  \le t_l \le 2 t_{n_k}}
\Big|\frac{t_{n_k} -t_l}{t_l}\Big| \to 0, \qquad k\to \infty.
\neqn
We will show that for an appropriate choice of $\{T_k\}$ either $S(z)$ or $(z-t_1)S(z)$ is the
desired function.

The sets $T_k$ will be chosen inductively.
Suppose that the sets $T_1, \dots T_{k-1}$ have been already chosen
and put $U_{k-1} = \prod_{j=1}^{k-1} S_{T_j}$. Then, clearly,
$|U_{k-1}(z)| \sim q_{k}|z|^{N_k}$, $|z|\to \infty$, for some constants $q_k>0$
and $N_k\in\mathbb{N}$.

Let us first consider the case when
$T_k = \{t_n:\, t_{n_k} /2  \le t_n \le 2 t_{n_k} \}$,
assuming that $n_k$ is sufficiently large. Then, by \eqref{tnk-prod},
the corresponding function $S_{T_k}$ satisfies
\[
t_{n_k}^2 |U_{k-1}(t_{n_k})
S'_{T_k} (t_{n_k})| = t_{n_k} |U_{k-1} (t_{n_k})|
\prod_{l \ne n_k:\,\, t_{n_k}/2  \le t_l \le 2 t_{n_k}}
\Big|\frac{t_{n_k} -t_l}{t_l}\Big| <1.
\]
In particular,
$$
\sum_{t_n\in T_k} \frac{1}{|t_n U_{k-1}(t_n) S_{T_k}'(t_n)|} \ge
t_{n_k}\gg 1.
$$

Now consider another extreme case where $T_k$
consists only of the point $t_{n_k}$,
that is, $T_k = \{t_{n_k}\}$.
Then $S_{T_k}(z)= 1-z/t_{n_k}$, and we have
\[
\sum_{t_n\in T_k} \frac{1}{|t_n U_{k-1}(t_n) S_{T_k}'(t_n)|} =\frac{1}{|U_{k-1}(t_{n_k})|} \ll 1.
\]
Hence, there exists a (not necessarily unique)
set $T_k \subset \{t_n:\, t_{n_k} /2  \le t_n \le 2 t_{n_k} \}$ such that
$S_{T_k}$ satisfies
\begin{equation}
\label{zlt}
\sum_{t_n\in T_k} \frac{1}{|t_n U_{k-1}(t_n) S_{T_k}'(t_n)|}>1
\end{equation}
and $T_k$ is minimal in the sense that the estimate \eqref{zlt}
no longer holds if one removes any point from $T_k$. This will be our choice of $T_k$.

Now let $t_j$ be any point in $T_k$ and let
$\tilde S_j = S_{T_k\setminus\{t_j\}}$. Then, by the
above property of minimality,
$$
\begin{aligned}
1\ge \sum_{t_n\in T_k, \, t_n\ne t_j}
\frac{1}{|t_n U_{k-1}(t_n) \tilde S_j'(t_n)|}
& =
\sum_{t_n\in T_k, \, t_n\ne t_j}
\frac{1}{|t_n U_{k-1}(t_n) S_{T_k}'(t_n)|}\cdot \frac{|t_j-t_n|}{|t_j|}
\\
& \gtrsim
\sum_{t_n\in T_k, \, t_n\ne t_j}
\frac{1}{t_n^2 |U_{k-1}(t_n) S_{T_k}'(t_n)|},
\end{aligned}
$$
where the last inequality follows from the hypothesis
$\inf_{n\ne j} |t_n-t_j| >0$. Since $t_j \in T_k$ was arbitrary,
we conclude that, uniformly with respect to $k$,
\begin{equation}
\label{zle}
\sum_{t_n\in T_k}
\frac{1}{t_n^2 |U_{k-1}(t_n) S_{T_k}'(t_n)|} \lesssim  1.
\end{equation}

Obviously, by choosing $t_{n_k}$ to grow sufficiently fast, we may achieve
that, for the function $S=\prod_k S_{T_k}$
the factors $S_{T_j}$ with $j>k$ almost do not influence the product
at the points $t_n\in T_k$ so that $\frac{1}{2} \le \prod_{j=k+1}^\infty
|S_{T_j}(t_n)| \le 2$ for $t_n \in [t_{n_k}/2 , 2t_{n_k}]$. Then
\begin{equation}
\label{zle1}
\frac{1}{2}
\sum_{t_n\in T_k}
\frac{1}{|t_n S'(t_n)|}
\le
\sum_{t_n\in T_k}
\frac{1}{ |t_n U_{k-1}(t_n) S_{T_k}'(t_n)|}
\le 2 \sum_{t_n\in T_k}
\frac{1}{|t_n S'(t_n)|},
\end{equation}
\begin{equation}
\label{zle2}
\frac{1}{2}
\sum_{t_n\in T_k}
\frac{1}{t_n^2 |S'(t_n)|}
\le
\sum_{t_n\in T_k}
\frac{1}{ t_n^2 |U_{k-1}(t_n) S_{T_k}'(t_n)|}
\le 2 \sum_{t_n\in T_k}
\frac{1}{t_n^2 |S'(t_n)|}.
\end{equation}
Also, it follows from Lemma \ref{bon} and from \eqref{zle} that
$\# T_k$, the number of elements in $T_k$,
satisfies $\# T_k \lesssim N_k \ln t_{n_k} + \ln q_k$.
Since $N_k$ and $q_k$ do not depend on the choice of $t_{n_k}$,
the function $S$  will be of zero order if $t_{n_k}$ grow sufficiently fast.

By the construction of $S_{T_k}$ (namely, by \eqref{zlt} and
\eqref{zle1}) we clearly have
$$
\sum_{t_n \in Z_S} \frac{1}{|t_n S'(t_n)|} =\infty.
$$
If, at the same time,
$$
\sum_{t_n \in Z_S} \frac{1}{t_n^2 |S'(t_n)|} < \infty,
$$
then our construction is completed. If the latter sum is also infinite, then
put $\tilde S=(z-t_1) S = (z-t_1) \prod_k S_{T_k}$. Then, clearly,
$|\tilde S'(t_n)| \asymp |t_n S'(t_n)|$, $t_n \in Z_S$,
and so, by \eqref{zle2} and \eqref{zle}, we have
$$
\sum_{t_n \in Z_S} \frac{1}{t_n^2 |\tilde S'(t_n)|}
\lesssim
\sum_k \frac{1}{t_{n_k}} \sum_{t_n\in T_k} \frac{1}{t_n^2|U_{k-1}(t_n) S'_{T_k}(t_n)|}
\lesssim \sum_k \frac{1}{t_{n_k}}<\infty.
$$
Thus, $\tilde S$ has the required properties.
\end{proof}

\subsection{End of the proof of Theorem \ref{hhn}}
\label{eqfor}
Let $S$ be the entire function constructed in Lemma \ref{hon}.
Then the function $G=A/S$ is entire.
The proof of \eqref{inter} follows by the standard interpolation series argument.
Indeed, the series in the right-hand side of \eqref{inter} converges absolutely
by the conditions on $S$.
Note that
$$
H(z) =
\frac{1}{S(z)} - \frac{1}{S(0)}
-\sum_{t_n \in Z_S} \frac{1}{S'(t_n)} \bigg(\frac{1}{z- t_n} +\frac{1}{t_n}\bigg)
$$
is an entire function (the poles disappear). Since $S$ is of order less
than one with real zeros, we conclude that $1/S$ is of Smirnov class in the upper and
in the lower half-planes, as well as the regularized Cauchy transform
in the right-hand side of
\eqref{inter}. Hence, by the classical
theorem of M.G. Krein (see, e.g., \cite[Part II, Chapter 1]{hj}),
$H$ is an entire function of zero exponential type. Note also that $|H(iy)|=o(|y|)$,
$|y| \to \infty$, whence $H$ is a constant. Since $H(0)=0$, we finally
get that $H\equiv 0$.

Thus, $G$ satisfies \eqref{glav}.
Put $a_n=|c_n|^{1/2}$, $b_n=c_n/|c_n|^{1/2}$.
Since $\sum_n |a_nb_nt_n^{-1}| =\infty$, we conclude that
$\sum_n |a_n|^2 =\sum_n |b_n|^2=\infty$.
By Theorem~D, the function $G$
corresponds to the rank one perturbation $\LL$ of $\A$,
generated by $\{a_n\}$, $\{b_n\}$ and $\varkappa = 1/S(0)$. Moreover,
$\LL^*$ also is a well-defined singular rank one perturbation
of $\A$ and the system of its eigenvectors  is unitary equivalent to the system
of reproducing kernels $\{K_\lambda\}_{\lambda \in Z_G}$ in $\he$.  It remains to see that
the latter system is not complete in $\he$.
However, it is a basic fact of the de Branges theory
that $\{K_\lambda\}_{\lambda \in Z_A}$ is an orthogonal basis of $\he$
(see \cite[Theorem 22]{br}). Hence,
$\{K_\lambda\}_{\lambda \in Z_A\setminus Z_S}$
is incomplete with infinite defect. Theorem \ref{hhn} is proved.
\qed

\

{\obeylines
Anton D. Baranov
Department of Mathematics and Mechanics, Saint Petersburg State University
28, Universitetski pr., St. Petersburg 198504, Russia
\vskip-.45cm
\phantom{r} and
National Research University  Higher School of Economics, St. Petersburg,
Russia, \enspace email: anton.d.baranov@gmail.com
}

\

\bigskip

{\obeylines
Dmitry V. Yakubovich
Departamento de Matem\'{a}ticas, Universidad Autonoma de Madrid
Cantoblanco 28049 (Madrid) Spain
\vskip-.45cm
\phantom{r} and
Instituto de Ciencias Matem\'{a}ticas (CSIC - UAM - UC3M - UCM) \enspace
email: dmitry.yakubovich@uam.es
}

\end{document}